\documentclass[10pt]{amsart}

\usepackage{amsfonts,amsmath,amssymb,amscd}
\usepackage{mathrsfs}
\theoremstyle{plain}
\newtheorem{thm}{Theorem}
\newtheorem{prop}[thm]{Proposition}
\newtheorem{lem}[thm]{Lemma}
\newtheorem{coro}[thm]{Corollary}

\newtheorem*{thm*}{Theorem}
\newtheorem*{prop*}{Proposition}
\newtheorem*{lem*}{Lemma}
\newtheorem*{coro*}{Corollary}
\newtheorem*{conj*}{Conjecture}
\newtheorem*{prethm*}{Pre-Theorem}
\newtheorem*{preprop*}{Pre-Proposition}
\newtheorem*{prelem*}{Pre-Lemma}
\newtheorem*{precoro*}{Pre-Corollary}
\newtheorem*{fact*}{Fact}
\newtheorem*{q*}{Question}
\newtheorem*{bigq*}{Big Question}

\newtheorem*{thm4}{Theorem 4}

\newtheorem*{prop*K}{Proposition K}

\newtheoremstyle{break}
  {9pt}
  {9pt}
  {\itshape}
  {}
  {\bfseries}
  {.}
  {\newline}
  {}

\theoremstyle{break}
\newtheorem{br-thm}{Theorem}
\newtheorem{br-prop}[thm]{Proposition}
\newtheorem{br-lem}[thm]{Lemma}
\newtheorem{br-coro}[thm]{Corollary}
\newtheorem{br-conj}[thm]{Conjecture}
\newtheorem{br-prethm}[thm]{Pre-Theorem}
\newtheorem{br-preprop}[thm]{Pre-Proposition}
\newtheorem{br-prelem}[thm]{Pre-Lemma}
\newtheorem{br-precoro}[thm]{Pre-Corollary}
\newtheorem{br-fact}[thm]{Fact}
\newtheorem{br-q}[thm]{Question}
\newtheorem{br-bigq}[thm]{Big Question}

\newtheoremstyle{plain+}
  {3pt}
  {3pt}
  {\itshape}
  {}
  {\bfseries}
  {.}
  {.5em}
  {\thmnote{#1 #2 #3}}

\theoremstyle{plain+}
\newtheorem{thm+}[thm]{Theorem}
\newtheorem{prop+}[thm]{Proposition}
\newtheorem{lem+}[thm]{Lemma}
\newtheorem{coro+}[thm]{Corollary}
\newtheorem{conj+}[thm]{Conjecture}
\newtheorem{q+}[thm]{Question}
\newtheorem{bigq+}[thm]{Big Question}

\newtheorem*{thm*+}{Theorem}
\newtheorem*{prop*+}{Proposition}
\newtheorem*{lem*+}{Lemma}
\newtheorem*{coro*+}{Corollary}
\newtheorem*{conj*+}{Conjecture}
\newtheorem*{q*+}{Question}
\newtheorem*{bigq*+}{Big Question}

\theoremstyle{definition}
\newtheorem{defn}[thm]{Definition}

\theoremstyle{remark}
\newtheorem*{rem}{Remark}
\newtheorem*{conv}{Convention}

\newtheorem*{claim*}{Claim}

\newtheorem{claim}{Claim}

\newcommand{\bbC}{\mathbb{C}}

\newcommand{\bbR}{\mathbb{R}}

\newcommand{\bbZ}{\mathbb{Z}}

\newcommand{\calA}{\mathcal{A}}

\newcommand{\calC}{\mathcal{C}}

\newcommand{\calG}{\mathcal{G}}

\newcommand{\calS}{\mathcal{S}}

\newcommand{\bfH}{\mathbf{H}}

\newcommand{\rmA}{\mathrm{A}}

\newcommand{\scrJ}{\mathscr{J}}

\newcommand{\scrL}{\mathscr{L}}

\newcommand{\scrT}{\mathscr{T}}

\newcommand{\nil}{\varnothing}

\newcommand{\Z}{\bbZ}

\newcommand{\R}{\bbR}
\newcommand{\C}{\bbC}

\newcommand{\Aut}{\operatorname{Aut}} 

\newcommand{\Out}{\operatorname{Out}} 
\newcommand{\Stab}{\operatorname{Stab}} 

\newcommand{\Homeo}{\operatorname{Homeo}}

\newcommand{\Fix}{\operatorname{Fix}}

\newcommand{\Mod}{\operatorname{Mod}}
\newcommand{\PMod}{\operatorname{PMod}}

\newcommand{\SL}{\operatorname{SL}}

\newcommand{\PSL}{\operatorname{PSL}}

\newcommand{\Teich}{\mathbf{Teich}}

\newcommand{\Hyp}{\bfH}

\newcommand{\PML}{\mathscr{P\!\!M\!L}}
\newcommand{\ML}{\mathscr{M\!L}}

\newcommand{\interior}{\mathrm{int}} 

\newcommand{\red}{\operatorname{\rho}}
\newcommand{\surface}{\Sigma}

\begin{document} 
%

%
%

\title[The girth alternative for $\Mod(\Sigma)$]{The girth alternative for \\ mapping class groups}
\author[K. Nakamura]{Kei Nakamura}
\address{Department of Mathematics \\ Temple University \\ Philadelphia, PA 19122}
\email{kei.nakamura@temple.edu}

\begin{abstract}
The \emph{girth} of a finitely generated group $\Gamma$ is the supremum of the girth of Cayley graphs for $\Gamma$ over all finite generating sets. Let $\Gamma$ be a finitely generated subgroup of the mapping class group $\Mod(\surface)$, where $\surface$ is a compact orientable surface. Then, either $\Gamma$ is virtually abelian or it has infinite girth; moreover, if we assume that $\Gamma$ is not infinite cyclic, these alternatives are mutually exclusive.
\end{abstract}

%
%

\maketitle

%
\section{Introduction}
\label{Introduction}
%

Let $\surface$ be a compact orientable surface, which is possibly disconnected. The \emph{mapping class group} of $\surface$, denoted by $\Mod(\surface)$, is the group of isotopy classes of $\Homeo^+(\surface)$, i.e. the group of orientation preserving homeomorphisms of $\surface$. Mapping class groups have been studied extensively in complex analysis, low-dimensional topology, and geometric group theory for more than a century.

A fascinating aspect of the mapping class groups is that they share many properties with lattices in semi-simple Lie groups. One analogy between linear groups and mapping class groups can be seen in the following famous dichotomy regarding their subgroups, which has become known as the \emph{Tits-alternative} for linear groups and mapping class groups respectively.

\begin{thm}[\cite{Tits:FreeSubgrp}] \label{tits-alt: linear} 
Let $\Bbbk$ be a field, and let $\Gamma$ be a finitely generated subgroup of $GL(n,\Bbbk)$. Then, $\Gamma$ either contains a non-abelian free subgroup or is virtually solvable; moreover, these alternatives are mutually exclusive.
\end{thm}

\begin{thm}[\cite{Ivanov:Tits-Margulis-Soifer}, \cite{McCarthy:Tits}] \label{tits-alt: mcg}
Let $\surface$ be a compact orientable surface, and let $\Gamma$ be a finitely generated subgroup of $\Mod(\surface)$. Then, $\Gamma$ either contains a non-abelian free subgroup or is virtually abelian; moreover, these alternatives are mutually exclusive.
\end{thm}

The dichotomy in the Tits alternative has been further investigated, and some refinements and variations are known for linear groups and mapping class groups; see, for example, the work of Margulis and So\u{\i}fer on maximal subgroups of linear groups \cite{Margulis--Soifer:MaximalSubgrp} and the analogous result by Ivanov for mapping class groups \cite{Ivanov:Tits-Margulis-Soifer}. The purpose of this article is to demonstrate that the structural analogy between these groups can be reinterpreted in terms of the \emph{girth} of finitely generated groups.

Recall that the \emph{girth} of a graph is the length of the shortest graph cycle, if any, in the graph. In \cite{Schleimer:Girth}, Schleimer defined the \emph{girth} of a finitely generated group $\Gamma$ to be the supremum of the girth of Cayley graphs of $\Gamma$ over all finite generating sets. By the observations of Schleimer in \cite{Schleimer:Girth} and the subsequent work of Akhmedov in \cite{Akhmedov:Girth1} and \cite{Akhmedov:Girth2}, there appears to be a significant qualitative difference between groups with finite girth and groups with infinite girth. In particular, Akhmedov gave the following ``\emph{girth alternative''} for linear groups, which shows that the division essentially coincides with the dichotomy in the Tits alternative for linear groups.

\begin{thm} [\cite{Akhmedov:Girth2}] \label{girth-alt: linear}
Let $\Bbbk$ be a field, and let $\Gamma$ be a finitely generated subgroup of $GL(n,\Bbbk)$. Then, $\Gamma$ is either a non-cyclic group with infinite girth or a virtually solvable group; moreover, these alternatives are mutually exclusive.
\end{thm}

Our main result is the following analogous girth alternative for mapping class groups $\Mod(\surface)$, showing that the division of subgroups into the ones with finite girth and the ones with infinite girth again coincides with the dichotomy in the Tits alternative.

\begin{thm} \label{girth-alt: mcg}
Let $\surface$ be a compact orientable surface, and let $\Gamma$ be a finitely generated subgroup of $\Mod(\surface)$. Then, $\Gamma$ is either a non-cyclic group with infinite girth or a virtually abelian group; moreover, these alternatives are mutually exclusive.
\end{thm}
For \emph{irreducible} subgroups of mapping class groups $\Mod(\surface)$, the result was independently proved by Yamagata \cite{Yamagata:Girth}; see \S 4 for the definition of irreducible subgroups.

One difficulty in proving Theorem~\ref{girth-alt: mcg} is that we cannot simply pass the statement to a finite index normal subgroup; unfortunately, we don't know that a group has infinite girth even if it has a finite index normal subgroup with infinite girth. This requires us to study the structure of $\Gamma$ carefully when $\surface$ is disconnected or when $\Gamma$ is \emph{reducible}.

In the general context of finitely generated groups, the existence of a non-abelian free subgroup in $\Gamma$ does not imply that $\Gamma$ has infinite girth. Hence, the girth alternative is not a mere consequence of the Tits alternative. As we shall see, for subgroups of mapping class groups, the girth alternative is a slightly more intricate manifestation of underlying structural properties that are responsible for the Tits alternative. A part of the arguments in the proof of Theorem~\ref{girth-alt: mcg} can also be used to show the girth alternative for convergence groups (see \cite{Nakamura:Thesis}, \cite{Yamagata:Girth}) and subgroups of $\Out(F_n)$ that contain an irreducible element with irreducible powers (see \cite{Nakamura:Thesis}).


\subsection{Acknowledgement}
The results in this paper are contained in the Ph.D. dissertation by the author, submitted in 2008 to University of California, Davis \cite{Nakamura:Thesis}. The author would like to thank Dmitry Fuchs, Misha Kapovich, and his advisor Joel Hass for being the members of the dissertation committee, carefully reading the original exposition of this work, and providing insightful comments while the author was at University of California, Davis. The author would also like to thank David Futer and Igor Rivin for encouraging him to make this work publicly available in the present format.

\subsection{Outline}
We present some results on the girth of finitely generated groups in \S\ref{Criteria}. The main tool used in the proof of Theorem~\ref{girth-alt: mcg} is the Infinite Girth Criterion (Proposition~\ref{igc}) in \S\ref{Infinite Girth}, which reformulates and generalizes the work of Akhmedov in \cite{Akhmedov:Girth2}; see \cite{Yamagata:Girth} for a similar alternative reformulation. The essence of the proof of this criterion is a reminiscent of the classical \emph{ping-pong argument}, which goes back to the work of Blaschke, Klein, Schottky, and Poincar\'e on Schottky groups in $\PSL(2,\R)$ and $\PSL(2,\C)$; see, for example, \cite{Klein:Ping-Pong}. In \S\ref{Elements} and \S\ref{Subgroups}, we review the properties of elements and subgroups of mapping class groups from Thurston's theory \cite{Thurston:Surfaces} and the proof of the Tits alternative \cite{Birman-Lubotzky-McCarthy:AbSol}, \cite{Ivanov:Tits-Margulis-Soifer}, \cite{McCarthy:Tits}, \cite{Ivanov:MCGBook}, and also prove a few facts that are necessary for the application of Infinite Girth Criterion in the proof of Theorem~\ref{girth-alt: mcg}. Finally, in \S\ref{girth alternative}, we present the proof of Theorem~\ref{girth-alt: mcg}.

\subsection{Conventions}
For various reasons, the mapping class group $\Mod(\surface)$ of a connected surface $\surface$ \emph{with boundary} is often defined in the literature to be the group of isotopy classes of $\Homeo^+(\surface,\partial \surface)$, consisting of homeomorphisms which take each component of $\partial \surface$ to itself; that is \emph{not} the convention we follow in this article. Our definition of $\Mod(\surface)$ as the group of isotopy classes of $\Homeo^+(\surface)$ implies that elements of $\Mod(\surface)$ may permute components of $\partial \surface$.

This convention coincides with the ones used in \cite{Birman-Lubotzky-McCarthy:AbSol}, \cite{McCarthy:Tits}, \cite{Ivanov:Tits-Margulis-Soifer}, and \cite{Ivanov:MCGBook}, where Thurston's theory on mapping class groups \cite{Thurston:Surfaces} was utilized in the studies of the structures of subgroups. One aspect of Thurston's theory involves a process of cutting the surface $\surface$ along an essential multi-loop which is invariant under the action of a mapping class and obtaining the induced mapping class on the resulted surface, say $\surface_{\calC}$; this new mapping class generally permutes the components of $\surface_{\calC}$ and $\partial \surface_{\calC}$, and the theory is developed most naturally under the convention we adopt in this article.

%
\section{Girth of Finitely Generated Groups}
\label{Criteria}
%

The \emph{girth} of a graph is the combinatorial length of the shortest cycle in the graph if there is a nontrivial cycle in the graph, and is set to be infinity if there is no cycle in the graph. Using the girth of Cayley graphs, Schleimer introduced in \cite{Schleimer:Girth} the notion of the \emph{girth} of a finitely generated group.

\begin{defn} \label{girth}
Let $\Gamma$ be a finitely generated group. Let $U(\Gamma, \calG)$ be the girth of the Cayley graph of $\Gamma$ with respect to a generating set $\calG$. The \emph{girth} of the group $\Gamma$ is defined to be $U(\Gamma):=\sup_\calG\{U(\Gamma,\calG)\}$, where the supremum is taken over all finite generating sets $\calG$ of $\Gamma$.
\end{defn}

\begin{conv}
Throughout the rest of the article, a group $\Gamma$ is assumed to be finitely generated unless otherwise stated.
\end{conv}

Clearly, every finite group has finite girth and any free group has infinite girth. It is also easy to see that an abelian group has finite girth unless it is the infinite cyclic group. In this section, we discuss some criteria for the girth of a group to be finite or infinite in a general setting. The Infinite Girth Criterion (Proposition~\ref{igc}) in \S\ref{Infinite Girth} is employed as the main tool in \S\ref{girth alternative}.

\subsection{Criteria for Finite Girth}
\label{Criteria: Finite Girth}

Recall that a group $\Gamma$ is said to satisfy a law if there is a word $w(x_1, \cdots , x_n)$ on $n$ letters such that $w(\gamma_1, \cdots , \gamma_n)=1$ in $\Gamma$ for any $\gamma_1, \cdots , \gamma_n \in \Gamma$. Schleimer obtained an important criterion for a group to have finite girth in \cite{Schleimer:Girth}.

\begin{thm}[\cite{Schleimer:Girth}] \label{law criteria}
A group $\Gamma$ has finite girth if it satisfies a law and is not infinite cyclic.
\end{thm}

Together with the proposition below, Theorem~\ref{law criteria} provides a significant portion of the class of groups that are known to have finite girth.

\begin{prop}[\cite{Schleimer:Girth}] \label{schleimer criteria}
A group $\Gamma$ has finite girth if it satisfies one of the following conditions:
\begin{itemize}
\item $\Gamma$ contains a finite-index subgroup with finite girth.
\item $\Gamma$ admits a finite-kernel surjection onto a group with finite girth.
\end{itemize}
\end{prop}

\begin{coro} \label{virtually solvable}
A virtually solvable group $\Gamma$ has finite girth, unless it is infinite cyclic.
\end{coro}

One may ask if amenable groups, other than the infinite cyclic one, always have finite girth. However, according to the work of Akhmedov in \cite[\S2]{Akhmedov:Girth1}, the answer turns out to be negative. We also note his observation in \cite[\S4]{Akhmedov:Girth1} that certain groups constructed by Olshanskii in his book \cite{Olshanskii:RelationsBook} have finite girth but do not satisfy any law. Characterizing groups with finite girth appears to be a delicate and difficult task. See \cite[\S5]{Akhmedov:Girth1} for related questions.

\subsection{Criteria for Infinite Girth}
\label{Infinite Girth}

The proof of the Tits alternative for linear groups and mapping class groups, as well as for other classes of groups, use variations of \emph{ping-pong} lemma \cite{Klein:Ping-Pong} to construct a free subgroup. The following formulation was given in \cite{Tits:FreeSubgrp}.

\begin{prop}[Free Subgroup Criterion, \cite{Tits:FreeSubgrp}] \label{fsc}
Let $\Gamma$ be a group acting on a set $X$. Suppose there exist elements $\sigma, \tau \in \Gamma$, subsets $U_\sigma, U_\tau \subset X$, and a point $x \in X$, such that
\begin{enumerate}
\item $\displaystyle{x \not \in U_\sigma \cup U_\tau}$,
\item $\displaystyle{\sigma^k ( \{x\} \cup U_\tau ) \subset U_\sigma }$ for all $k \in \Z-\{0\}$, and
\item $\displaystyle{\tau^k ( \{x\} \cup U_\sigma) \subset U_\tau }$ for all $k \in \Z-\{0\}$.
\end{enumerate}
Then, $\langle \sigma, \tau \rangle$ is a non-abelian free subgroup of $\Gamma$.
\end{prop}

\begin{rem}
For the proof of the above criterion, one merely observes inductively that any nontrivial reduced word in $\sigma^{\pm 1}$ and $\tau^{\pm 1}$ takes $x \in X-(U_\sigma \cup U_\tau)$ into $U_\sigma \cup U_\tau$ via the action of $\langle \sigma, \tau \rangle$, showing that the word cannot represent the identity element of $\Gamma$. This is the classical \emph{ping-pong argument}.
\end{rem}

In the study of the girth of groups, a criterion for a group to have infinite girth is instrumental. Generalizing and reformulating the work of Akhmedov \cite{Akhmedov:Girth2}, we give such a criterion in comparable generality as Proposition~\ref{fsc}; see \cite{Yamagata:Girth} for a similar alternative reformulation.

\begin{prop}[Infinite Girth Criterion] \label{igc}
Let $\Gamma$ be a group acting on a set $X$, with a finite generating set $\calG:=\{\gamma_1, \cdots , \gamma_n\}$. Suppose there exist elements $\sigma, \tau \in \Gamma$, subsets $U_\sigma, U_\tau \subset X$, and a point $x \in X$, such that
\begin{enumerate}
\item $\displaystyle{x \not \in (U_\sigma \cup U_\tau) \cup \bigcup_{\varepsilon=\pm1} \bigcup_{i=1}^n \gamma_i^\varepsilon(U_\sigma \cup U_\tau)}$,
\item $\displaystyle{\sigma^k \bigg( \{x\} \cup U_\tau \cup \bigcup_{\varepsilon=\pm1} \bigcup_{i=1}^n \gamma_i^\varepsilon(U_\tau) \bigg) \subset U_\sigma }$ for all $k \in \Z-\{0\}$, and
\item $\displaystyle{\tau^k \bigg( \{x\} \cup U_\sigma \cup \bigcup_{\varepsilon=\pm1} \bigcup_{i=1}^n \gamma_i^\varepsilon(U_\sigma) \bigg) \subset U_\tau }$ for all $k \in \Z-\{0\}$.
\end{enumerate}
Then, $\Gamma$ is a non-cyclic group with infinite girth.
\end{prop}

\begin{rem}
Clearly, $\sigma$, $\tau$, $U_\sigma$, $U_\tau$, and $x \in X$ in Proposition~\ref{igc} satisfy the conditions in Proposition~\ref{fsc}. Hence, $\langle \sigma, \tau \rangle < \Gamma$ must be a non-abelian free subgroup.
\end{rem}

\begin{proof} 
Let $M$ be a positive integer, and we aim to find a new generating set $\hat{\calG}$ for $\Gamma$ such that $U(\Gamma, \hat{\calG}\, ) \geq M$. Let $P=\{ p_1, \cdots , p_n \}$ be a set of positive intergers such that $p_i>M$ for all $i$ and $|p_i-p_j|>M$ for all distinct $i, j$. Let $\hat \gamma_i:= \sigma^{p_i}\gamma_i \tau^{-p_i}$ for each $i$. The set $\hat{\calG}:=\{\sigma, \tau, \hat \gamma_1, \cdots , \hat \gamma_n \}$ clearly generates $\Gamma$. Let $w$ be a nontrivial reduced word in $\hat{\calG} \cup \hat{\calG}^{-1}$ with the length less than $M$ with respect to $\hat{\calG} \cup \hat{\calG}^{-1}$. We can write $w$ as
\[
w=u_1 \hat \gamma_{i_1}^{\varepsilon_1} u_2 \hat \gamma_{i_2}^{\varepsilon_2} \cdots \hat \gamma_{i_s}^{\varepsilon_s} u_{s+1}\\
\]
where $\varepsilon_\ell \in \{ \pm1 \}$ and the subword $u_\ell=u_\ell(\sigma, \tau)$ is a (possibly empty) reduced word in $\{ \sigma^{\pm 1}, \tau^{\pm 1} \}$. If $u_\ell$ is an empty word and $i_{\ell-1} = i_\ell$ for some $\ell$, we must have $\varepsilon_{\ell-1}=\varepsilon_\ell$. For otherwise, a cancellation occurs and contradicts the assumption that $w$ is a reduced word in $\hat{\calG} \cup \hat{\calG}^{-1}$.

Now, regarded as an element of $\Gamma$, $w$ can be expressed as
\begin{eqnarray*}
w&=&u_1 \hat \gamma_{i_1}^{\varepsilon_1} u_2 \hat \gamma_{i_2}^{\varepsilon_2} \cdots \hat \gamma_{i_s}^{\varepsilon_s} u_{s+1}\\
&=&v_1 \gamma_{i_1}^{\varepsilon_1} v_2 \gamma_{i_2}^{\varepsilon_2} \cdots \gamma_{i_s}^{\varepsilon_s} v_{s+1}
\end{eqnarray*}
where $v_\ell=v_\ell(\sigma, \tau)$ is a reduced word in $\{ \sigma^{\pm 1}, \tau^{\pm 1} \}$ for $\beta_{\ell-1} u_\ell \alpha_\ell$ (with convention $\alpha_{s+1}=\beta_0=1$) and
\[
\alpha_\ell=\begin{cases}
\sigma^{p_{i_\ell}} & \text{if $\varepsilon_\ell=+1$} \\
\tau^{p_{i_\ell}} & \text{if $\varepsilon_\ell=-1$}
\end{cases}
\;\;\;\;\;\;
\beta_\ell=\begin{cases}
\tau^{-p_{i_\ell}} & \text{if $\varepsilon_\ell=+1$} \\
\sigma^{-p_{i_\ell}} & \text{if $\varepsilon_\ell=-1$}
\end{cases}
\]
The idea of the proof is to apply the ping-pong argument to $w$ to show that $w$ cannot represent the identity element of $\Gamma$. Provided with suitable initial points in $X$, the ping-pong argument applies easily to the strings $v_\ell$. What we need to show is that we can pass each $\gamma^{\varepsilon_\ell}_{i_\ell}$ in the ping-pong argument; in other words, we need to check that $\gamma^{\varepsilon_\ell}_{i_\ell}$ takes the terminal point from the ping-pong rally $v_{\ell+1}$ to a suitable initial point for the ping-pong rally $v_\ell$. We will see that our choice of $p_i$ prevents excessive cancellations, and we can indeed pass each $\gamma^{\varepsilon_\ell}_{i_\ell}$ under the conditions (2) and (3) in the statement of the proposition.

\begin{claim} \label{pass gamma}
For each $\ell$, $v_\ell$ is not an empty word. If $\varepsilon_\ell=+1$, then the last letter of $v_\ell$ is $\sigma^{\pm 1}$ and the first letter of $v_{\ell+1}$ is $\tau^{\pm 1}$. If $\varepsilon_\ell=-1$, then the last letter of $v_\ell$ is $\tau^{\pm 1}$ and the first letter of $v_{\ell+1}$ is $\sigma^{\pm 1}$.
\end{claim}

\begin{proof}[Proof of Claim~\ref{pass gamma}] 
Let us show that, if $\varepsilon_\ell=+1$, then $v_\ell$ is a non-empty word ending with $\sigma^{\pm1}$. Since $\varepsilon_\ell=+1$, we have $\alpha_\ell= \sigma^{p_{i_\ell}}$ and $\beta_\ell= \tau^{-p_{i_\ell}}$. There are three cases to consider: (i) $u_\ell$ is an empty word; (ii) the last letter of $u_\ell$ is $\tau^{\pm 1}$; or (iii) the last letter of $u_\ell$ is $\sigma^{\pm 1}$.

Case (i): If $u_\ell$ is empty, then $v_\ell$ is the reduced word for $\beta_{\ell-1} \alpha_\ell$, and thus we have
\[
v_\ell=\begin{cases}
\sigma^{p_{i_\ell}} & \text{if $\ell=1$} \\
\tau^{-p_{i_{\ell-1}}} \sigma^{p_{i_\ell}} & \text{if $\ell \neq 1$ and $\varepsilon_{\ell-1}=+1$} \\
\sigma^{-p_{i_{\ell-1}}+p_{i_\ell}} & \text{if $\ell \neq 1$ and $\varepsilon_{\ell-1}=-1$} \\
\end{cases}
\]
In the last subcase, since $\varepsilon_\ell=+1 \neq -1=\varepsilon_{\ell-1}$, we must have $i_\ell \neq i_{\ell-1}$ as noted before. Thus, we must have $|p_{i_\ell}-p_{i_{\ell-1}}|>M$, and it follows that $v_\ell$ is a nontrivial power of $\sigma$. In all subcases, $v_\ell$ is indeed a non-empty word ending with $\sigma^{\pm 1}$.

Case (ii): If the last letter of $u_\ell$ is $\tau^{\pm 1}$, then there is no cancellation between $u_\ell$ and $\alpha_\ell= \sigma^{p_{i_\ell}}$ as a word in $\{ \sigma^{\pm 1}, \tau^{\pm 1} \}$. Hence, $v_\ell$ is again a non-empty word ending with $\sigma^{\pm 1}$.

Case (iii): Finally, suppose the last letter of $u_\ell$ is $\sigma^{\pm 1}$. If $u_\ell$ is not a power of $\sigma$, then $u_\ell= \cdots \tau^{q} \sigma^{p}$ for some $q$ and $p$. So, $u_\ell \alpha_\ell= \cdots \tau^q \sigma^{p+p_{i_\ell}}$. Note that we must have $|p| < M$. For otherwise, the length of $u_\ell$ as a word in $\hat{\calG} \cup \hat{\calG}^{-1}$, and thus the length of $w$ as a word in $\hat{\calG} \cup \hat{\calG}^{-1}$, is at least $M$; this contradicts with the assumption on the length of $w$. Now, $|p|<M$ and $p_{i_\ell}>M$ together imply $p+p_{i_\ell} \neq 0$. Thus, the last letter of $v_\ell$ must be $\sigma^{\pm 1}$. If $u_\ell$ is a power of $\sigma$, say $u_\ell= \sigma^p$, then
\[
v_\ell=\begin{cases}
\sigma^{p+p_{i_\ell}} & \text{if $\ell=1$} \\
\tau^{-p_{i_{\ell-1}}} \sigma^{p+p_{i_\ell}} & \text{if $\ell \neq 1$ and $\varepsilon_{\ell-1}=+1$} \\
\sigma^{-p_{i_{\ell-1}}+p+p_{i_\ell}} & \text{if $\ell \neq 1$ and $\varepsilon_{\ell-1}=-1$}
\end{cases}
\]
In the first two subcases, $v_\ell$ ends with a nontrivial power of $\sigma$, because $|p|<M$ and $p_{i_\ell} >M$ imply $p+p_{i_\ell} \neq 0$. In the third subcase, we must have $i_\ell \neq i_{\ell-1}$, and hence $|p_{i_\ell}-p_{i_{\ell-1}}|>M$. It now follows from $|p|<M$ that $-p_{i_{\ell-1}}+p+p_{i_\ell} \neq 0$, and $v_\ell$ is again a nontrivial power of $\sigma$. Thus, in all subcases, $v_\ell$ is again a non-empty word ending with $\sigma^{\pm 1}$ as desired.

This concludes the proof that, if $\varepsilon_\ell=+1$, then $v_\ell$ is a non-empty word ending with $\sigma^{\pm 1}$. The analogous arguments show that, if $\varepsilon_\ell=+1$, then $v_{\ell+1}$ is a non-empty word beginning with $\tau^{\pm 1}$. The symmetric arguments show that, if $\varepsilon_\ell=-1$, then $v_\ell$ is a nonempty word ending with $\tau^{\pm1}$ and $v_{\ell+1}$ is a non-empty word beginning with $\sigma^{\pm1}$.
\end{proof} 


\begin{claim} \label{nontrivial} 
If the last letter of $v_{s+1}$ is $\sigma^{\pm1}$ and $y \in \{x\} \cup U_\tau \cup \bigcup_{\varepsilon=\pm1} \bigcup_{i=1}^n \gamma^{\varepsilon}_i(U_\tau)$, or if the last letter of $v_{s+1}$ is $\tau^{\pm1}$ and $y \in \{x\} \cup U_\sigma \cup \bigcup_{\varepsilon=\pm1} \bigcup_{i=1}^n \gamma^{\varepsilon}_i(U_\sigma)$, then $w(y) \in U_\sigma \cup U_\tau$.
\end{claim}

\begin{proof}[Proof of Claim~\ref{nontrivial}] 
We will prove the claim by induction on $s$. If $s=0$, $w=v_1$ is merely a reduced word in $\sigma^{\pm1}$ and $\tau^{\pm1}$. In this case, $w(y) \in U_\sigma \cup U_\tau$ follows from the classical ping-pong argument as in the proof of Free Subgroup Criterion.

Now, as the induction hypothesis, suppose that the claim is true for $s-1 \geq 0$, and let $w=v_1 \gamma_{i_1}^{\varepsilon_1} v_2 \gamma_{i_2}^{\varepsilon_2} \cdots \gamma_{i_s}^{\varepsilon_s} v_{s+1}$. Suppose $\varepsilon_s=+1$ for now, so that the first letter of $v_{s+1}$ is $\tau^{\pm1}$ by Claim~\ref{pass gamma}. Then, we have $v_{s+1}(y) \in U_\tau$ by the classical ping-pong argument, and we obtain $y':=\gamma_{i_s}v_{s+1}(y) \in \gamma_{i_s}(U_\tau)$. Now, also by Claim~\ref{pass gamma}, the last letter of $v_s$ is $\sigma^{\pm1}$. Thus, applying the induction hypothesis to $w':=v_1 \gamma_{i_1}^{\varepsilon_1} v_2 \gamma_{i_2}^{\varepsilon_2} \cdots \gamma_{i_{s-1}}^{\varepsilon_{s-1}} v_s$ and $y'$, we see that $w(y)=w'(y') \in U_\sigma \cup U_\tau$. The $\varepsilon_s=-1$ case is analogous.
\end{proof} 

Since $x \not \in U_\sigma \cup U_\tau$ by the assumption and $w(x) \in U_\sigma \cup U_\tau$ by Claim~\ref{nontrivial}, it follows that $w$ cannot represent the identity element in $\Gamma$. Namely, any non-empty word in $\hat{\calG}$ that represents the identity element of $\Gamma$ must be of length at least $M$ with respect to $\hat{\calG}$. Hence, $U(\Gamma) \geq U(\Gamma; \hat{\calG}\,) \geq M$.
\end{proof} 

\begin{rem}
Although every group that satisfies our Infinite Girth Criterion (Proposition~\ref{igc}) must contain a non-abelian free subgroup, generally groups containing non-abelian free subgroup need \emph{not} have infinite girth. The groups constructed by Olshanskii \cite{Olshanskii:RelationsBook}, which we mentioned in \S\ref{Criteria: Finite Girth}, contain non-abelian free subgroups and have finite girth; indeed, the existence of such non-abelian free subgroups is precisely the reason why those groups do not satisfy any law.
\end{rem}

It is natural to ask how the property of having infinite girth behaves under homomorphisms. The following partial answer by Akhmedov \cite{Akhmedov:Girth2} plays an important role in \S4.

\begin{prop}[\cite{Akhmedov:Girth2}] \label{akhmedov criterion}
A group $\Gamma$ has infinite girth if it admits a surjection onto a non-cyclic group with infinite girth.
\end{prop}

As noted in \S\ref{Introduction}, we do not know if the presence of a finite-index infinite-girth subgroup of $\Gamma$ guarantees that $\Gamma$ itself has infinite girth; indeed, it is suspected that is not the case in general (see \cite{Akhmedov:Girth2}). This is unfortunate, because we cannot pass to a finite-index subgroup to prove that a given group has infinite girth. In our proof of the girth alternative for subgroups of mapping class groups, we go around this obstacle by applying the above proposition in a suitable manner.

%
\section{Elements of Mapping Class Groups}
\label{Elements}
%

Let $\surface$ be a (not necessarily connected) compact orientable surface. A \emph{multi-loop} on $\surface$ is a pair-wise disjoint collection of simple closed curves on $\surface$, and it is said to be \emph{essential} if each component is not null-homotopic or peripheral.
\begin{itemize}
\item An isotopy class $\calA$ of an essential (possibly empty) multi-loop on $\surface$ is said to be a \emph{reduction system} for $\sigma \in \Mod(\surface)$ if $\sigma$ fixes $\calA$.
\item An element $\sigma \in \Mod(\surface)$ is said to be \emph{reducible} if it admits a \emph{non-empty} reduction system $\calA$; it is said to be \emph{irreducible} otherwise.
\item An element $\sigma \in \Mod(\surface)$ is said to be \emph{periodic} if its order is finite; it is said to be \emph{aperiodic} otherwise.
\end{itemize}
Nielsen studied $\Mod(T^2) \cong \SL(2,\Z)$ in terms of the dynamical properties of its action on the Teichm\"{u}ller space $\Teich(T^2) \cong \Hyp^2$ and its boundary \cite{Nielsen:MCG}, and he essentially showed that irreducible elements are precisely the \emph{Anosov elements}, i.e. the elements that can be represented by \emph{Anosov homeomorphisms} of $T^2$. The far-reaching generalization of Nielsen's work, dealing with surfaces $\surface$ that admit complete hyperbolic metrics, was developed by Thurston \cite{Thurston:Surfaces}. Among other things, Thurston introduced the notion of \emph{pseudo-Anosov homeomorphisms} of $\surface$ and showed that irreducible elements of $\Mod(\surface)$ are precisely the \emph{pseudo-Anosov elements}, i.e. the elements that can be represented by pseudo-Anosov homeomorphisms of $\surface$.


\subsection{The Canonical Reduction of Reducible Elements}
\label{Can-Red: Elts} 

The notion of the the \emph{reduction} of an element of $\Mod(\surface)$ appeared in Thurston's theory, and it was further studied in \cite{Birman-Lubotzky-McCarthy:AbSol}. If $\calA$ is a reduction system for $\tau \in \Mod(\surface)$ and $\surface_{\calA}$ is the compactification of $\surface-\calA$, then $\tau$ induces a mapping class $\red_{\!\calA}(\tau) \in \Mod(\surface_{\calA})$; this element $\red_{\!\calA}(\tau)$ is called the \emph{reduction of $\tau$ along $\calA$}. Although the reduction is meant to be applied to reducible elements along non-empty reduction system, we allow a reduction system $\calA$ of an element $\tau$ to be empty for the logical convenience; if $\calA=\nil$, the reduction $\red_{\!\calA}(\tau)$ coincides with $\tau$.

Given an element $\tau \in \Mod(\surface)$, we can always take some positive power $\tau^N$ so that $\tau^N$ takes each connected component of $\surface$ to itself and each component of $\partial \surface$ to itself. If the restrictions of $\tau^N$ to some components are periodic, we can take yet higher power $\tau^{N'}$ so that the restriction of $\tau^{N'}$ to each component is either trivial or aperiodic.

\begin{itemize}
\item An element $\tau \in \Mod(\surface)$ is \emph{adequately reduced} if there is some power $\tau^N$, taking each component of $\surface$ to itself and each component of $\partial \surface$ to itself, such that the restriction of $\tau^N$ to each component of $\surface$ is either (i) trivial or (ii) aperiodic and irreducible.
\item Given an element $\tau \in \Mod(\surface)$, a (possibly empty) reduction system $\calA$ of $\tau$ is said to be an \emph{adequate reduction system} if the reduction $\red_{\!\calA}(\tau)$ is adequately reduced.
\end{itemize}

By definition, an empty reduction system for an adequately reduced element is indeed an adequate reduction system. Thurston observed that every element $\tau \in \Mod(S)$ either is adequately reduced or has a non-empty adequate reduction system \cite{Thurston:Surfaces}.

Generally, adequate reduction systems of $\tau$ are not unique. The crucial fact shown in \cite{Birman-Lubotzky-McCarthy:AbSol} is that there is a canonical choice of an adequate reduction system for each element $\tau \in \Mod(\surface)$; it is indeed the unique minimal adequate reduction system for the element. We call such system the \emph{canonical reduction system} for $\tau$, and denote it by $\calC$ (the reference to $\tau$ should always be clear from the context).

For a connected surface $\surface$, the canonical reduction system for $\tau \in \Mod(\surface)$ is empty if and only if $\tau$ is periodic or irreducible. More generally, for any surface $\surface$, which may be connected or disconnected, the canonical reduction system for $\tau \in \Mod(\surface)$ is empty if and only if $\tau$ is adequately reduced. Adequately reduced elements are more general than irreducible ones, but their important properties can be derived from those of irreducible ones.


\subsection{Pseudo-Anosov Elements for Connected Surfaces}
\label{Pseudo-Anosov: Conn}

Suppose for now that $\surface$ is a compact orientable \emph{connected} surface such that the interior of $\surface$ admits a complete hyperbolic metric, i.e. $\chi(\surface)<0$. We write $g(\surface)$ and $b(\surface)$ for the genus and the number of connected components of $\partial \surface$ respectively. Thurston introduced the space $\PML(\surface)$ of \emph{projective measured laminations}, equipped with a topology which makes it homeomorphic to a sphere of dimension $6g(\surface)+2b(\surface)-7$. With respect to this topology, the mapping class group $\Mod(\surface)$ acts naturally by homeomorphisms. $\PML(\surface)$ compactifies the Teichm\"{u}ller space of $\surface$, coherently with respect to the actions of $\Mod(\surface)$.

An element $\sigma \in \Mod(\surface)$ is said to be \emph{pseudo-Anosov} if the fixed-point set $\Fix(\sigma) \subset \PML(\surface)$ of the action of $\sigma$ on $\PML(\surface)$ consists of a pair of distinct measured laminations. The key property of a pseudo-Anosov element $\sigma$ is that its action on $\PML(\surface)$ exhibits the \emph{north-south dynamics} with one of the fixed points as an attractor, denoted by $\scrL^+_\sigma$, and the other as a repeller, denoted by $\scrL^-_\sigma$. More precisely, for any pair of disjoint neighborhoods $U^+_\sigma$ of $\scrL_\sigma^+$ and $U^-_\sigma$ of $\scrL_\sigma^-$, we have $\sigma^{\pm N} \big( \PML(\surface)-U^\mp_\sigma \big) \subset U^\pm_\sigma$ respectively for all sufficiently large $N$. 


We remark that pseudo-Anosov elements are always aperiodic. Furthermore, as mentioned earlier, one of the main results in Thurston's work \cite{Thurston:Surfaces} is that irreducible elements in $\Mod(\surface)$ are precisely pseudo-Anosov ones; see \cite{FLP:Thurston} for the detail.


\subsection{Pseudo-Anosov Elements for Disonnected Surfaces}
\label{Pseudo-Anosov: Disconn}

Let us now allow a compact orientable surface $\surface=\bigsqcup_{i=1}^{c(\surface)}\surface^i$ to be \emph{disconnected}, where $c(\surface)$ denotes the number of connected components of $\surface$ and $\surface^i$ denotes each connected component. As in the case of connected surfaces, we will assume that the interior of $\surface$ admits a complete hyperbolic metric, i.e. $\chi(\surface^i)<0$ for each $\surface^i$.

Every element $\sigma \in \Mod(\surface)$ has a nontrivial power that takes each connected component $\surface^i$ to itself. An element of $\Mod(\surface)$ is said to be \emph{pseudo-Anosov} if the restriction of such a power to each component $\surface^i$ is a pseudo-Anosov element in $\Mod(\surface^i)$. It is easy to see that this notion of pseudo-Anosov elements is well-defined, and they are necessarily aperiodic. Furthermore, one can check that irreducible elements of $\Mod(\surface)$ are precisely pseudo-Anosov ones. 

With some care, the space $\PML(\surface)$ of projective measure laminations can be defined. As shown by Ivanov \cite{Ivanov:Tits-Margulis-Soifer} and McCarthy \cite{McCarthy:Tits}, the space $\PML(\surface)$ turns out to be homeomorphic to the \emph{join} of $\PML(\surface^i)$. There is a natural action of $\Mod(\surface)$ on $\PML(\surface)$, but the description of this action is rather cumbersome. For example, the action of pseudo-Anosov element $\sigma \in \Mod(\surface)$ still exhibits dynamical properties similar to the connected case, with an attracting neighborhood $U^+_\sigma$ and a repelling neighborhood $U^-_\sigma \subset \PML(\surface)$, but neither of them can arise as a neighborhood of a single projective measured lamination any more; under the homeomorphism $\PML(\surface) \cong \ast_{i=1}^{c(\surface)} \PML(\surface^i)$, certain simplexes in the join $\ast_{i=1}^{c(\surface)} \PML(\surface^i)$ play the roles of the attractor and the repeller.

For our purposes, it is convenient to look at the action of $\Mod(\surface)$ on an alternative space. Following Ivanov \cite{Ivanov:MCGBook}, we look at the natural action of $\Mod(\surface)$ on the space
\[
\PML^\sharp(\surface):=\bigsqcup_{i=1}^{c(\surface)} \PML(\surface^i)
\]
instead. If an element $\sigma \in \Mod(\surface)$ takes each component $\surface^i \subset \surface$ to itself, it takes each component $\PML(\surface^i) \subset \PML^\sharp(\surface)$ to itself under this action. If $\sigma$ is further assumed to be pseudo-Anosov, then its action on each component $\PML(\surface^i) \subset \PML^\sharp(\surface)$ exhibits the honest north-south dynamics with fixed points $\scrL^\pm_{\sigma|_{\surface^i}} \in \PML(\surface^i)$. Moreover, the points $\scrL^\pm_{\sigma|_{\surface^i}}$ are precisely the fixed points for the action of $\sigma$ on $\PML^\sharp(\surface)$. For any pair of disjoint neighborhoods $U^+_\sigma$ of $\big\{\scrL^+_{\sigma|_{\surface^i}} \big\}_{i=1}^{c(\surface)}$ and $U^-_\sigma$ of $\big\{\scrL^-_{\sigma|_{\surface^i}} \big\}_{i=1}^{c(\surface)}$, we have $\sigma^N \big( \PML^\sharp(\surface)-U^-_\sigma \big) \subset U^+_\sigma$ and $\sigma^{-N} \big( \PML^\sharp(\surface)-U^+_\sigma \big) \subset U^-_\sigma$ for all sufficiently large $N$.

%
\section{Subgroups of Mapping Class Groups}
\label{Subgroups}
%

Ivanov further generalized Thurston's theory to subgroups of $\Mod(\surface)$ in order to study their structures \cite{Ivanov:MCGBook}, and obtained a classification of subgroups that parallels the classification of elements of $\Mod(\surface)$.
\begin{itemize}
\item An isotopy class $\calA$ of an essential (possibly empty) multi-loop is said to be a \emph{reduction system} for $\Gamma$ if it is a reduction system for every element $\sigma \in \Gamma$, i.e. if every element $\sigma \in \Gamma$ fixes $\calA$.
\item A subgroup $\Gamma$ is said to be \emph{reducible} if it admits a \emph{non-empty} reduction system $\calA$; it is said to be \emph{irreducible} otherwise.
\end{itemize}
By definition, an element $\sigma \in \Mod(\surface)$ is reducible (resp. irreducible) if and only if the cyclic subgroup $\langle \sigma \rangle < \Mod(\surface)$ is reducible (resp. irreducible). More generally, the above definitions suggest that (i) the analogue of periodic elements are finite subgroups, (ii) the analogue of aperiodic reducible elements are infinite reducible subgroups, and (iii) the analogue of pseudo-Anosov elements are infinite irreducible subgroups.

\subsection{Torsion-Free Finite-Index Subgroups}
\label{Mod-m}

Before we proceed, let us first consider a useful family of subgroups of $\Mod(\surface)$. For each integer $m \geq 3$, we consider the natural homomorphisms
\[
\Mod(\surface) \rightarrow \Aut(H_1(\surface;\Z)) \rightarrow \Aut(H_1(\surface;\Z/m\Z))
\]
and let $\Mod_{(m)}(\surface)$ be the kernel of this composition of homomorphisms. $\Mod_{(m)}(\surface)$ is clearly a finite-index normal subgroup of $\Mod(\surface)$, and a classical theorem of Serre \cite{Serre:Rigidite} says that $\Mod_{(m)}(\surface)$ is torsion-free. For each subgroup $\Gamma < \Mod(\surface)$, we set $\Gamma_{(m)}:=\Gamma \cap \Mod_{(m)}(\surface)$; it is a torsion-free finite-index normal subgroup of $\Gamma$.

Ivanov observed that elements of $\Mod_{(m)}(\surface)$ possess other useful properties. For every element $\sigma \in \Mod_{(m)}(\surface)$ and for every reduction system $\calA$ of $\sigma$, the following statements hold \cite[\S1.2]{Ivanov:MCGBook}:
\begin{itemize}
\item $\sigma$ takes each component of $\surface$ to itself, and each component of $\partial \surface$ to itself;
\item $\sigma$ takes each component of $\calA$ to itself with its orientation preserved;
\item hence, the reduction $\red_{\!\calA}(\sigma) \in \Mod(\surface_{\calA})$ takes each component of $\surface_{\calA}$ to itself, and each component of $\partial \surface_{\calA}$ to itself.
\end{itemize}
It follows that $\sigma$ can be restricted to each component of $\surface$, and $\red_{\!\calA}(\sigma)$ can be restricted to each component of $\surface_{\calA}$. Such restrictions have the following properties, which is much stronger than the aperiodicity of $\sigma \in \Mod_{(m)}(\surface)$ \cite[\S1.6]{Ivanov:MCGBook}:
\begin{itemize}
\item the restriction of $\sigma$ to each component of $\surface$ is trivial or aperiodic;
\item the restriction of $\red_{\!\calA}(\sigma)$ to each component of $\surface_{\calA}$ is trivial or aperiodic.
\end{itemize}
As an immediate consequence, we have the following observation: for any subgroup $\Gamma < \Mod(\surface)$, the normal subgroup $\Gamma_{(m)} \lhd$ can be restricted to each component of $\surface$, and such restrictions are torsion-free.



The properties of $\Gamma_{(m)}$ mentioned in the above paragraphs are used extensively in Ivanov's proof of the analogue of the Tits alternative. In the proof of the analogue of \emph{Margulis--So\u{\i}fer theorem}, which is a significantly stronger theorem than the Tits alternative, Ivanov employed additional properties of the subgroup $\Gamma_{(m)}$ and its relationship to $\Gamma$ \cite[\S9]{Ivanov:MCGBook}. We extract one of such properties as the following lemma.

\begin{lem}[{c.f. \cite[\S9.10]{Ivanov:MCGBook}}]
\label{partition}
Let $\surface=\bigsqcup_{i=1}^{c(\surface)} \surface^i$ be a surface which is not necessarily connected. For every component $\surface^i$ and every element $\sigma \in \Gamma$, the restrictions of $\Gamma_{(m)}$ to $\surface^i$ and $\sigma(\surface^i)$ are isomorphic. Hence, if the restriction of $\Gamma_{(m)}$ to components $\surface^i$ and $\surface^j$ are not isomorphic, then no element of $\Gamma$ can take $\surface^i$ to $\surface^j$.
\end{lem}

This elementary fact did not play any role in proving the Tits alternative for $\Mod(\surface)$, while it serves as a critical step in proving the Margulis--So\u{\i}fer theorem for $\Mod(\surface)$. The proof of our main theorem too make use of the above lemma. 

\begin{proof} 
The proof we present here is contained in \cite[\S9.10]{Ivanov:MCGBook} where the lemma was stated for a special case. Since $\Gamma_{(m)} \lhd \Gamma$, the conjugate of an element $\tau \in \Gamma_{(m)}$ by an element $\sigma \in \Gamma$ must again belong to $\Gamma_{(m)}$; if $s$ and $t$ are homeomorphisms representing $\sigma$ and $\tau$ respectively, then $s \circ t \circ s^{-1}$ represents the element $\sigma \tau \sigma^{-1} \in \Gamma_{(m)}$. The restriction $\tau|_{\surface^i}$ is represented by $t|_{\surface^i}$, and the restriction $\sigma \tau \sigma^{-1}|_{\sigma(\surface^i)}$ is represented by $s \circ (t|_{\surface^i}) \circ s^{-1}$. Thus, the conjugation by $\sigma$ defines a homomorphism that takes the restriction $\Gamma_{(m)}|_{\surface^i}$ to the restriction $\Gamma_{(m)}|_{\sigma(\surface^i)}$. Clearly, this homomorphism is an isomorphism, with the inverse given by the conjugation by $\sigma^{-1}$.
\end{proof}

\subsection{Canonical Reduction of Reducible Subgroups}
\label{Can-Red: Subgrps}

The theory regarding the \emph{reduction} of a subgroup of $\Mod(\surface)$ can be developed in essentially the same manner as it was done for an element of $\Mod(\surface)$, with suitable modifications. If $\calA$ is a reduction system for $\Gamma<\Mod(\surface)$, then the reduction $\red_{\!\calA}(\sigma) \in \Mod(\surface_{\calA})$ is well-defined for all $\sigma \in \Mod(\surface)$, where $\surface_{\calA}$ is the compactification of $\surface-\calA$ as before. The assignment $\sigma \mapsto \red_{\!\calA}(\sigma)$ indeed defines the \emph{reduction homomorphism}
\[
\red_{\!\calA}: \Gamma \rightarrow \Mod(\surface_{\calA})
\]
whose kernel is a free-abelian group generated by Dehn twists along some components of $\calA$; the image $\red_{\!\calA}(\Gamma)$ is called the \emph{reduction of $\Gamma$ along $\calA$}. As before, a reduction system $\calA$ of a subgroup $\Gamma$ is allowed to be empty; if $\calA=\nil$, the reduction $\Gamma_\calA$ coincides with $\Gamma$.

Recall from \S\ref{Can-Red: Elts} that, for an element $\tau \in \Mod(\surface)$, taking a finite power $\tau^N$ was essential in understanding the reduction systems of $\tau$. Ivanov observed that, for a subgroup $\Gamma < \Mod(\surface)$, the correct analogue is passing to a finite-index normal subgroup $\Gamma' \lhd \Gamma$. To make this analogy apparent, we introduce the following notions, which did not appear explicitly in Ivanov's exposition:

\begin{itemize}
\item A subgroup $\Gamma < \Mod(\surface)$ is \emph{adequately reduced} if there is a \emph{finite-index normal subgroup} $\Gamma' \lhd \Gamma$, consisting of elements that take each component of $\surface$ to itself and each component of $\partial \surface$ to itself, such that the restriction of $\Gamma'$ to each component is either (i) trivial or (ii) infinite and irreducible. 
\item Given a subgroup $\Gamma < \Mod(\surface)$, a (possibly empty) reduction system $\calA$ of $\Gamma$ is said to be an \emph{adequate reduction system} if the reduction $\red_{\!\calA}(\Gamma)$ is adequately reduced.
\end{itemize}

The work of Ivanov shows that for any subgroup $\Gamma < \Mod(\surface)$ there is a canonical choice of reduction system \cite[\S7.2-7.4]{Ivanov:MCGBook}, which is indeed the unique minimal adequate reduction system \cite[\S7.16 and \S7.18]{Ivanov:MCGBook}; we call this system the \emph{canonical reduction system} for $\Gamma$, and denote it by $\calC$ (the reference to $\Gamma$ should always be clear from the context). Although we will not go into the detail of the definition of the canonical reduction system, we note that the definition is invariant under passing to finite-index normal subgroup.

\begin{lem}
\label{adeq-red and Gamma_m}
Let $\surface=\bigsqcup_{i=1}^{c(\surface)}\surface^i$ be a compact orientable surface. For any subgroup $\Gamma < \Mod(\surface)$, the following are equivalent:
\begin{itemize}
\item[(1)] the canonical reduction system $\calC$ for $\Gamma$ is empty;
\item[(2)] $\Gamma$ is adequately reduced;
\item[(3)] some finite-index normal subgroups of $\Gamma$ are adequately reduced;
\item[(4)] every finite-index normal subgroup of $\Gamma$ is adequately reduced;
\item[(5)] for some integers $m \geq 3$, the restriction of $\Gamma_{(m)}$ to each component of $\surface$ is either (i) trivial or (ii) infinite and irreducible;
\item[(6)] for every integer $m \geq 3$, the restriction of $\Gamma_{(m)}$ to each component of $\surface$ is either (i) trivial or (ii) infinite and irreducible.
\end{itemize}
\end{lem}

\begin{proof}
(1) $\Leftrightarrow$ (2) follows from the fact that $\calC$ is the unique minimal adequate reduction system. (3) $\Rightarrow$ (2) and (2) $\Rightarrow$ (4) follows from the invariance of $\calC$ under passing to finite-index normal subgroup and the equivalence (1) $\Leftrightarrow$ (2). With the trivial implication (4) $\Rightarrow$ (3), we have the equivalence of (1)--(4).

Since (6) $\Rightarrow$ (5) is trivial and (5) $\Rightarrow$ (2) follows from the definition, it remains to check the implication (2) $\Rightarrow$ (6). Suppose $\Gamma$ is adequately reduced, and take an integer $m \geq 3$. By the equivalence (2) $\Leftrightarrow$ (4), $\Gamma_{(m)}$ is adequately reduced, i.e. there exists a finite-index normal subgroup $\Gamma' \lhd \Gamma_{(m)}$ such that the restriction of $\Gamma'$ to each component is either trivial or irreducible. Recall that the restriction of $\Gamma_{(m)}$ to each component is torsion-free; hence, if the restriction of $\Gamma'$ to a component $\surface^i$ is trivial, then the restriction of $\Gamma_{(m)}$ must also be trivial. If the restriction of $\Gamma'$ to a component $\surface^j$ is irreducible, then clearly the restriction of the larger group $\Gamma_{(m)}$ to $\surface^j$ must also be irreducible. This completes the proof.
\end{proof}

\subsection{Adequately Reduced Subgroups for Connected Surfaces}
\label{Adeq-Red: Conn}

Let $\surface$ be a compact orientable \emph{connected} surface with $\chi(\surface)<0$, and let $\Gamma < \Mod(\surface)$ be an adequately reduced subgroup. $\Gamma$ is either a finite subgroup or an infinite irreducible subgroup. An infinite irreducible subgroup always contains a pseudo-Anosov element \cite[\S5.9 and \S7.14]{Ivanov:MCGBook}. Furthermore, if $\Gamma$ is infinite, irreducible, and not virtually infinite-cyclic, then $\Gamma_{(m)} \lhd \Gamma$ contains two pseudo-Anosov elements $\sigma$ and $\tau$ such that $\Fix(\sigma)\cap\Fix(\tau)=\nil$ in $\PML(\surface)$ \cite[\S5.12 and \S7.15]{Ivanov:MCGBook}. These facts were proved with the following lemmas, which we cite here for reference:

\begin{lem}[{\cite[\S5.1]{Ivanov:MCGBook}}]
\label{constructing pseudo-Anosov for conn}
Let $\Gamma<\Mod(\surface)$, and suppose that $\sigma, \tau \in \Gamma_{(m)}$ are pseudo-Anosov elements such that $\Fix(\sigma) \cap \Fix(\tau)=\nil$ in $\PML(\surface)$. Then, the elements $\sigma^p\tau^q$ is pseudo-Anosov for all sufficiently large $p$ and $q$.
\end{lem}

\begin{lem} [{\cite[\S5.11]{Ivanov:MCGBook}}]
\label{fixed-point sets of pseudo-Anosov}
Let $\Gamma<\Mod(\surface)$, and suppose that $\sigma, \tau \in \Gamma_{(m)}$ are pseudo-Anosov elements. Then, $\Fix(\sigma)=\Fix(\tau)$ or $\Fix(\sigma)\cap\Fix(\tau)=\nil$ in $\PML(\surface)$.
\end{lem}

In order to study the girth of adequately reduced subgroups using the Infinite Girth Criterion, we need a slightly stronger fact than Ivanov's results. The following theorem summarizes and refines Ivanov's results.

\begin{thm} \label{trichotomy: adeq-red for conn}
Fix an integer $m \geq 3$. Let $\surface$ be a connected surface with $\chi(\surface)<0$, and let $\Gamma < \Mod(\surface)$ be an adequately reduced subgroup. Then, $\Gamma$ satisfies exactly one of the following:
\begin{itemize}
\item[(0)] $\Gamma_{(m)} \lhd \Gamma$ is trivial, and $\Gamma$ is finite;
\item[(1)] $\Gamma_{(m)} \lhd \Gamma$ is infinite-cyclic, and $\Gamma$ is virtually infinite-cyclic;
\item[(2)] For any (possibly empty) finite collection $\varphi_1, \cdots , \varphi_n \in \Gamma_{(m)} \lhd \Gamma$ of pseudo-Anosov elements, there exists another pseudo-Anosov element $\psi \in \Gamma_{(m)} \lhd \Gamma$ such that $\Fix(\varphi_j)\cap\Fix(\psi)=\nil$ in $\PML(\surface)$ for all $j$.
\end{itemize}
\end{thm}

\begin{rem}
There exist finite irreducible subgroups which consist entirely of reducible elements \cite{Gilman:FiniteIrreducible}.
\end{rem}

\begin{proof}
Since $\surface$ is connected, Lemma~\ref{adeq-red and Gamma_m} says that $\Gamma_{(m)}$ is either trivial or irreducible. If $\Gamma_{(m)}$ is trivial, $\Gamma$ is finite; if $\Gamma_{(m)}$ is infinite-cyclic, then $\Gamma$ is virtually infinite-cyclic.
Hence, we assume $\Gamma_{(m)}$ is infinite, irreducilbe, and not infinite-cyclic. As we have already mentioned, there exist two pseudo-Anosov elements $\sigma$ and $\tau$ in $\Gamma_{(m)}$ such that $\Fix(\sigma)\cap\Fix(\tau)=\nil$ in $\PML(\surface)$ \cite[\S5.12]{Ivanov:MCGBook}. We need to show that the slightly stronger condition (2) is satisfied.

Suppose we are given a finite collection $\varphi_1, \cdots , \varphi_n \in \Gamma_{(m)} \lhd \Gamma$ of pseudo-Anosov elements. For each $\varphi_j$, we let $U^+_j, U^-_j \subset \PML(\surface)$ be attracting and repelling neighborhoods of $\varphi_j$ such that $U^+_j \cap U^-_j=\nil$; we set $U_j=U^+_j \cup U^-_j$. Similarly, we let $U^+_\sigma, U^-_\sigma, U^+_\tau, U^-_\tau \subset \PML(\surface)$ be attracting and repelling neighborhoods of $\sigma$, $\tau$ respectively such that $U^+_\sigma \cap U^-_\sigma=\nil$ and $U^+_\tau \cap U^-_\tau=\nil$; again, we set $U_\sigma=U^+_\sigma \cup U^-_\sigma$ and $U_\tau=U^+_\tau \cup U^-_\tau$. We can take these neighborhoods to be sufficiently small so that the following holds:
\begin{itemize}
\item $U_j \cap U_k=\nil$ whenever $\Fix(\varphi_j) \cap \Fix(\varphi_k)=\nil$;
\item $U_j \cap U_\sigma=\nil$ whenever $\Fix(\varphi_j) \cap \Fix(\sigma)=\nil$;
\item $U_j \cap U_\tau=\nil$ whenever $\Fix(\varphi_j) \cap \Fix(\tau)=\nil$;
\item $U_\sigma \cap U_\tau=\nil$ (since $\Fix(\sigma) \cap \Fix(\tau)=\nil$).
\end{itemize}

By the property of pseudo-Anosov elements, there exists a large enough integer $M$ such that $\sigma^{\pm m}(\PML(\surface)-U^\mp_\sigma) \subset U^\pm_\sigma$ and $\tau^{\pm m}(\PML(\surface)-U^\mp_\tau) \subset U^\pm_\tau$ for all $m$ with $m>M$. By Lemma~\ref{constructing pseudo-Anosov for conn}, $\psi=\sigma^p \tau^q$ is pseudo-Anosov for all sufficiently large $p$ and $q$. Hence, taking such $p$ and $q$ greater than $M$, we obtain a pseudo-Anosov element $\psi=\sigma^p\tau^q$ such that $\psi(\PML(\surface)-U^-_\tau) \in U^+_\sigma \subset U_\sigma$ and $\psi^{-1}(\PML(\surface)-U^+_\sigma) \in U^-_\tau \subset U_\tau$. Furthermore, by the disjointness of $U_\sigma$ and $U_\tau$, we see that $\psi^\ell(\PML(\surface)-U^-_\tau) \in U^+_\sigma \subset U_\sigma$ and $\psi^{-\ell}(\PML(\surface)-U^+_\sigma) \in U^-_\tau \subset U_\tau$ for all positive integer $\ell$.

Choose $\scrL \in \PML(\surface)-(U_\sigma \cup U_\tau)$ and consider a sequence $\scrL^\ell:=\psi^\ell(\scrL)$, $\ell \in \Z$. Note that $\scrL^\ell \in U^+_\sigma$ and $\scrL^{-\ell} \in U^-_\tau$ for all $\ell>0$. Since $\psi$ is pseudo-Anosov, we must have a convergence $\scrL^\ell \rightarrow \scrL^+_\psi$ and $\scrL^{-\ell} \rightarrow \scrL^-_\psi$ as $\ell \rightarrow \infty$. Thus, we see that $\scrL^+_\psi \in \overline{U^+_\sigma} \subset \overline{U_\sigma}$ and $\scrL^-_\psi \in \overline{U^-_\tau} \subset \overline{U_\tau}$.

We claim that this $\Fix(\psi) \cap \Fix(\varphi_j)=\nil$ for each $j$. We first note that the assumption $\Fix(\sigma) \cap \Fix(\tau)=\nil$ and Lemma~\ref{fixed-point sets of pseudo-Anosov} imply that we have either $\Fix(\varphi_j) \cap \Fix(\sigma) =\nil$ or $\Fix(\varphi_j) \cap \Fix(\tau) = \nil$ for each $j$. By the choice of our neighborhoods $U_j$, $U_\sigma$, $U_\tau$, we thus have $U_j \cap \overline{U_\sigma}=\nil$ or $U_j \cap \overline{U_\tau}=\nil$ for each $i$. It follows that $\scrL^+_\psi \not \in \Fix(\varphi_j)$ or $\scrL^-_\psi \not \in \Fix(\varphi_j)$. Thus, by Lemma~\ref{fixed-point sets of pseudo-Anosov}, we conclude that $\Fix(\psi) \cap \Fix(\varphi_j)=\nil$ for each $j$.
\end{proof}

\subsection{Adequately Reduced Subgroups for Disconnected Surfaces}
\label{Adeq-Red: Disconn}

We now allow a compact orientable surface $\surface=\bigsqcup_{i=1}^{c(\surface)} \surface^i$ to be \emph{disconnected} surface, where $c(\surface)$ denotes the number of component of $\surface$ and $\surface^i$ denotes each component as before. Note that it is possible to have an infinite irreducible subgroup $\Gamma < \Mod(\surface)$ such that it restricts to some component $\surface^i$ as a finite irreducible subgroup of $\Mod(\surface^i)$; there is no hope in finding a pseudo-Anosov element in such $\Gamma$.

It turns out that this is essentially the only obstacle to finding a pseudo-Anosov element in $\Gamma$. Ivanov showed that, if $\Gamma_{(m)} \lhd \Gamma$ is also irreducible, i.e. the restriction of $\Gamma_{(m)}$ to each component is irreducible, then $\Gamma_{(m)} \lhd \Gamma$ contains a pseudo-Anosov element \cite[\S6.3]{Ivanov:MCGBook}. Furthermore, when the restriction of $\Gamma_{(m)}$ to \emph{every} component of $\surface$ is larger than an infinite-cyclic group, he showed that $\Gamma_{(m)} \lhd \Gamma$ contains two pseudo-Anosov elements $\sigma$ and $\tau$ such that $\Fix(\sigma) \cap \Fix(\tau)=\nil$ in $\PML^\sharp(\surface)$ \cite[\S6.4]{Ivanov:MCGBook}.

We need the analogue of Theorem~\ref{trichotomy: adeq-red for conn} for disconnected surfaces; this will be given below as Theorem~\ref{trichotomy: adeq-red for disconn} after a couple of lemmas. Generally, an adequately reduced group $\Gamma$ is a hybrid of three cases that appeared in Theorem~\ref{trichotomy: adeq-red for conn}.

\begin{lem}[{c.f. \cite[\S9.10]{Ivanov:MCGBook}}]
\label{partition: adeq-red}
Fix an integer $m \geq 3$, and let $\Gamma < \Mod(\surface)$ be an adequately reduced subgroup. Consider a partition $\surface=\surface^{[0]} \sqcup \surface^{[1]} \sqcup \surface^{[2]}$ defined by the following:
\begin{itemize}
\item[{[0]}] the subsurface $\surface^{[0]}$ is the union of all components $\surface^i$ such that the restriction of $\Gamma_{(m)}$ to $\surface^i$ is trivial;
\item[{[1]}] the subsurface $\surface^{[1]}$ is the union of all components $\surface^i$ such that the restriction of $\Gamma_{(m)}$ to $\surface^i$ is infinite-cyclic;
\item[{[2]}] the subsurface $\surface^{[2]}$ is the union of all components $\surface^i$ such that the restriction of $\Gamma_{(m)}$ to $\surface^i$ is nontrivial and non-cyclic.
\end{itemize}
Then, every element $\sigma \in \Gamma$ preserves this partition, i.e. $\sigma(\surface^{[\ell]})=\surface^{[\ell]}$ for $\ell=0,1,2$; hence, the restriction of $\Gamma$ to $\surface^{[\ell]}$ is well-defined for $\ell=0,1,2$.
\end{lem}

\begin{proof}
This follows immediately from Lemma~\ref{partition}.
\end{proof}

\begin{lem}
\label{two subgroups}
Fix an integer $m \geq 3$. For each $\ell$, let $\Gamma^{[\ell]}$ and $(\Gamma_{(m)})^{[\ell]}$ be the restrictions of $\Gamma$ and $\Gamma_{(m)}$ to $\surface^{[\ell]}$ respectively, and set $(\Gamma^{[\ell]})_{(m)}=\Gamma^{[\ell]} \cap \Mod_{(m)}(\surface^{[\ell]})$. Then, $(\Gamma_{(m)})^{[\ell]}$ is a finite-index normal subgroup of $(\Gamma^{[\ell]})_{(m)}$.
\end{lem}

\begin{proof}
By definition, an element of $(\Gamma^{[\ell]})_{(m)}$ is a restriction of an element of $\Gamma$ that acts trivially on $H_1(\surface^{[\ell]}; \Z/m\Z)<H_1(\surface; \Z/m\Z)$. An element of $(\Gamma_{(m)})^{[\ell]}$ clearly satisfies this property since it is a restriction of an element of $\Gamma_{(m)}$ that acts trivially on the entire $H_1(\surface; \Z/m\Z)$. Hence, $(\Gamma_{(m)})^{[\ell]}<(\Gamma^{[\ell]})_{(m)}<\Gamma^{[\ell]}$.

Since $\Gamma_{(m)}$ is a finite-index normal subgroup of $\Gamma$, we see that $(\Gamma_{(m)})^{[\ell]}$ must be a finite-index normal subgroup of $\Gamma^{[\ell]}$. Hence, we conclude that $(\Gamma_{(m)})^{[\ell]}$ must also be a finite-index normal subgroup of the intermediate subgroup $(\Gamma^{[\ell]})_{(m)}$.
\end{proof}

\begin{rem}
In general, $(\Gamma_{(m)})^{[\ell]}$ is contained in $(\Gamma^{[\ell]})_{(m)}$ as a \emph{proper} subgroup of $(\Gamma^{[\ell]})_{(m)}$. We also note that both $(\Gamma_{(m)})^{[\ell]}$ and $(\Gamma^{[\ell]})_{(m)}$ are torsion-free.
\end{rem}


\begin{thm}
\label{trichotomy: adeq-red for disconn}
Fix an integer $m \geq 3$, and let $\Gamma < \Mod(\surface)$ be an adequately reduced subgroup. Suppose $\surface=\surface^{[0]} \sqcup \surface^{[1]} \sqcup \surface^{[2]}$ is the partition given in Lemma~\ref{partition: adeq-red}, and set $\Gamma^{[\ell]}$ and $(\Gamma^{[\ell]})_{(m)}$ as in Lemma~\ref{two subgroups}. Then, $\Gamma$ satisfies all of the following: 
\begin{itemize}
\item[(0)] $(\Gamma^{[0]})_{(m)} \lhd \Gamma^{[0]}$ is trivial, and $\Gamma^{[0]}$ is finite;
\item[(1)] $(\Gamma^{[1]})_{(m)} \lhd \Gamma^{[1]}$ is free-abelian, and $\Gamma^{[1]}$ is virtually free-abelian;
\item[(2)] for any (possibly empty) finite collection $\varphi_1, \cdots, \varphi_n \in (\Gamma^{[2]})_{(m)} \lhd \Gamma^{[2]}$ of pseudo-Anosov elements, there exists another pseudo-Anosov element $\psi \in (\Gamma^{[2]})_{(m)} \lhd \Gamma^{[2]}$ such that $\Fix(\varphi_i) \cap \Fix(\psi)=\nil$ in $\PML^\sharp(\surface^{[2]})$ for all $i$.
\end{itemize}
\end{thm}

\begin{proof}
It follows from the choice of $\surface^{[0]}$ in Lemma~\ref{partition: adeq-red} that $(\Gamma_{(m)})^{[0]}$ is trivial; hence, $(\Gamma^{[0]})_{(m)}$ must also be trivial, and $\Gamma^{[0]}$ must be finite. To see that $(\Gamma^{[1]})_{(m)}$ is free-abelian, we consider its restriction to each component of $\surface^{[1]}$. Since the restriction of $(\Gamma_{(m)})^{[1]}<(\Gamma^{[1]})_{(m)}$ to each component is infinite-cyclic by definition, the restriction of $(\Gamma^{[1]})_{(m)}$ to each component must be virtually infinite-cyclic. Theorem~\ref{trichotomy: adeq-red for conn} then implies that this restriction of $(\Gamma^{[1]})_{(m)}$ to each component must be infinite-cyclic. Hence, $(\Gamma^{[1]})_{(m)}$ must be free-abelian, and $\Gamma^{[1]}$ is virtually free-abelian.

Now, we consider $\Gamma^{[2]}$ and $(\Gamma^{[2]})_{(m)}$. Note that it suffices to prove the statement (2) under the assumption $\surface=\surface^{[2]}$, which allow us to reduce the cluttering of the notations. With this assumption, we have $\Gamma=\Gamma^{[2]}$ and $\Gamma_{(m)}=(\Gamma^{[2]})_{(m)}$. The restriction of $\Gamma_{(m)}$ to each component of $\surface$ is nontrivial and non-cyclic, and $\Gamma_{(m)}$ is infinite and irreducible by Lemma~\ref{adeq-red and Gamma_m}. Hence, by Ivanov's result \cite[\S6.4]{Ivanov:MCGBook}, we know that there exists two pseudo-Anosov elements $\sigma$ and $\tau$ in $\Gamma_{(m)}$ such that $\Fix(\sigma)\cap\Fix(\tau)=\nil$ in $\PML^\sharp(\surface)$. We need to show that the following slightly stronger statement holds: for any finite collection $\varphi_1, \cdots, \varphi_n \in \Gamma_{(m)} \lhd \Gamma$ of pseudo-Anosov elements, there exists another pseudo-Anosov element $\psi \in \Gamma_{(m)} \lhd \Gamma$ such that $\Fix(\varphi_j) \cap \Fix(\psi)=\nil$ in $\PML^\sharp(\surface)$ for all $j$.

Recall that the group $\Gamma_{(m)}$ can be restricted to each component $\surface^i$ of $\surface$; in particular, the restrictions of pseudo-Anosov elements $\sigma, \tau, \varphi_j \in\Gamma_{(m)}$ to each component $\surface^i$ are again pseudo-Anosov.
The action of $\Gamma_{(m)}$ on $\PML^\sharp(\surface)$ descends to the action of the restriction of $\Gamma_{(m)}$ to $\surface^i$ on the component $\PML(\surface^i)$. The proof of Theorem~\ref{trichotomy: adeq-red for conn} applies to each component, and we will combine them to construct the desired $\psi \in \Gamma_{(m)}$.

Suppose we are given a finite collection $\varphi_1, \cdots, \varphi_n \in \Gamma_{(m)} \lhd \Gamma$ of pseudo-Anosov elements. For each $\varphi_j$ and each component $\surface^i$, we let $U^{+,i}_j, U^{-,i}_j \subset \PML(\surface^i)$ be attracting and repelling neighborhoods of $\varphi_j|_{\surface^i}$ such that $U^{+,i}_j \cap U^{-,i}_j=\nil$; we set $U^i_j=U^{+,i}_j \cup U^{-,i}_j$. Similarly, we let $U^{+,i}_\sigma, U^{-,i}_\sigma, U^{+,i}_\tau, U^{-,i}_\tau \subset \PML(\surface^i)$ be attracting and repelling neighborhoods of $\sigma|_{\surface^i}$, $\tau|_{\surface^i}$ respectively such that $U^{+,i}_\sigma \cap U^{-,i}_\sigma=\nil$ and $U^{+,i}_\tau \cap U^{-,i}_\tau=\nil$; again, we set $U^i_\sigma=U^{+,i}_\sigma \cup U^{-,i}_\sigma$ and $U^i_\tau=U^{+,i}_\tau \cup U^{-,i}_\tau$. We can take these neighborhoods to be sufficiently small so that the following holds:
\begin{itemize}
\item $U^i_j \cap U^i_k=\nil$ whenever $\Fix(\varphi_j|_{\surface^i}) \cap \Fix(\varphi_k|_{\surface^i})=\nil$;
\item $U^i_j \cap U^i_\sigma=\nil$ whenever $\Fix(\varphi_j|_{\surface^i}) \cap \Fix(\sigma|_{\surface^i})=\nil$;
\item $U^i_j \cap U^i_\tau=\nil$ whenever $\Fix(\varphi_j|_{\surface^i}) \cap \Fix(\tau|_{\surface^i})=\nil$;
\item $U^i_\sigma \cap U^i_\tau=\nil$ (since $\Fix(\sigma|_{\surface^i}) \cap \Fix(\tau|_{\surface^i})=\nil$ by assumption).
\end{itemize}

By the property of pseudo-Anosov elements, there exists a large enough integer $M$ such that $(\sigma|_{\surface^i})^{\pm m}(\PML(\surface^i)-U^{\mp,i}_\sigma) \subset U^{\pm,i}_\sigma$ and $(\tau|_{\surface^i})^{\pm m}(\PML(\surface^i)-U^{\mp,i}_\tau) \subset U^{\pm,i}_\tau$ for all $i$ and for all $m>M$. Also, by Lemma~\ref{constructing pseudo-Anosov for conn}, there exists large enough integers $P$ and $Q$ such that $\sigma^p\tau^q|_{\surface^i}=(\sigma|_{\surface^i})^p (\tau|_{\surface^i})^q$ is pseudo-Anosov for all $i$ and for all $p>P$ and $q>Q$; in other words, $\psi=\sigma^p \tau^q$ is pseudo-Anosov on $\surface$ for all $p>P$ and $q>Q$. Taking $p$ and $q$ greater than $M$, we obtain a pseudo-Anosov element $\psi=\sigma^p\tau^q$ such that $\psi|_{\surface^i}(\PML(\surface^i)-U^{-,i}_\tau) \in U^{+,i}_\sigma \subset U^i_\sigma$ and $\psi^{-1}|_{\surface^i}(\PML(\surface^i)-U^{+,i}_\sigma) \in U^{-,i}_\tau \subset U^i_\tau$. Furthermore, by the disjointness of $U^i_\sigma$ and $U^i_\tau$, we see that $\psi^\ell|_{\surface^i}(\PML(\surface^i)-U^{-,i}_\tau) \in U^{+,i}_\sigma \subset U^i_\sigma$ and $\psi^{-\ell}|_{\surface^i}(\PML(\surface^i)-U^{+,i}_\sigma) \in U^{-,i}_\tau \subset U^i_\tau$ for all $i$ and for all positive integer $\ell$.

Choose $\scrL \in \PML(\surface^i)-(U^i_\sigma \cup U^i_\tau)$ and consider a sequence $\scrL^\ell:=\psi^\ell(\scrL)$, $\ell \in \Z$. By the same arguments as in the proof of Theorem~\ref{trichotomy: adeq-red for conn}, we see that $\scrL^\ell \rightarrow \scrL^{+,i}_\psi \in \overline{U^{+,i}_\sigma} \subset \overline{U^i_\sigma}$ and $\scrL^{-\ell} \rightarrow \scrL^-_\psi \in \overline{U^{-,i}_\tau} \subset \overline{U^i_\tau}$ as $\ell \rightarrow \infty$, and we deduce that $\Fix(\psi|_{\surface^i}) \cap \Fix(\varphi_j|_{\surface^i})=\nil$. Since this is true for each component $\surface^i$, we conclude that $\Fix(\psi) \cap \Fix(\varphi_j)=\nil$ in $\PML^\sharp(\surface)$.
\end{proof}

%
\section{Girth Alternative}
\label{girth alternative}
%

We now consider the girth of subgroups $\Gamma$ of a mapping class group $\Mod(\surface)$, where $\surface$ is a compact surface; we do not a priori assume that $\surface$ is connected. Our main result is that the dichotomy between the subgroups with infinite girth and the ones with finite girth indeed coincides with the structural dichotomy of the Tits alternative shown in \cite{Ivanov:Tits-Margulis-Soifer} and \cite{McCarthy:Tits}. In other words, we have the following \emph{girth alternative}:

\begin{thm4}
Let $\surface$ be a compact orientable surface, and let $\Gamma$ be a finitely generated subgroup of $\Mod(\surface)$. Then, $\Gamma$ is either a non-cyclic group with infinite girth or a virtually free-abelian group; moreover, these alternatives are mutually exclusive.
\end{thm4}

The girth alternative above reduces to the case where the interior of $\surface$ admits a complete hyperbolic metric.

\begin{thm} \label{girth-alt: mcg hyp}
Let $\surface$ be a compact orientable surface whose interior admits a complete hyperbolic metric, and let $\Gamma$ be a finitely generated subgroup of $\Mod(\surface)$. Then, $\Gamma$ is either a non-cyclic group with infinite girth or a virtually free-abelian group; moreover, these alternatives are mutually exclusive.
\end{thm}

Let us first show that the general case, Theorem~\ref{girth-alt: mcg} in \S\ref{Introduction}, follows from the hyperbolic case, Theorem~\ref{girth-alt: mcg hyp} above.

\begin{proof}[Proof of Theorem~\ref{girth-alt: mcg}]
Suppose for now that $\surface$ consists of copies of tori and a (possibly disconnected) surface that admits a complete hyperbolic metric. Since the mapping class groups of tori and one-punctured tori are isomorphic, we may replace the copies of tori in $\surface$ with the same number of copies of once-punctured tori. In turn, we now realize $\Gamma$ as a subgroup of the mapping class group of a hyperbolic surface; Theorem~\ref{girth-alt: mcg} follows from Theorem~\ref{girth-alt: mcg hyp} as desired.

For the general case, let $\surface=\surface' \sqcup \surface''$, such that $\surface'$ is the union of tori and hyperbolic components, and that $\surface''$ is the union of spheres, disks, and annuli. If the restriction $\Gamma'$ of $\Gamma$ to $\surface'$ is a non-cyclic group with infinite girth, then so is $\Gamma$ by Proposition~\ref{akhmedov criterion}. So, assume that $\Gamma'$ is virtually abelian and let $A' < \Gamma'$ be a finite-index abelian subgroup.

Recall that the mapping class groups are trivial for the sphere, the disk, and the annulus; hence, the restriction of $\Gamma$ to $\surface''$ can only permute these components. Hence, the kernel $K$ of the restriction to $\surface''$ is a finite-index normal subgroup of $\Gamma$. Let $K'$ be the restriction of $K$ to $\surface'$; note that $K' \cap A'$ is a finite-index abelian subgroup of $K'$. Now, since elements of $K$ acts trivially on $\surface''$, we see that $K \cong K'$. Hence, $K$ contains a finite-index abelian subgroup isomorphic to $K' \cap A'$. As $K$ itself has finite index in $\Gamma$, we conclude that $\Gamma$ contains a finite-index abelian subgroup.
\end{proof}

Throughout the rest of this section, we will assume that the surface $\surface$ admits a complete hyperbolic metric. The proof for Theorem~\ref{girth-alt: mcg hyp} is the content of \S\ref{Girth: Adeq-Red} and \S\ref{Girth: Red}, and will be given in the form of Propositions~\ref{girth-alt: adeq-red for conn}, \ref{girth-alt: adeq-red for disconn}, and \ref{girth-alt: red}. The main idea of the proof is that a pair of generators of free subgroups in the Tits alternative can be carefully chosen so that Infinite Girth Criterion (Proposition~\ref{igc}) can be applied to these elements. Proposition~\ref{girth-alt: adeq-red for conn}, stated in a slightly different language, was independently proved by Yamagata \cite{Yamagata:Girth}.

\subsection{Girth of Adequately Reduced Subgroups}
\label{Girth: Adeq-Red}

We first consider an adequately reduced subgroup $\Gamma < \Mod(\surface)$, where $\surface$ is a \emph{connected} surface.

\begin{prop}[c.f. \cite{Yamagata:Girth}] \label{girth-alt: adeq-red for conn}
Let $\surface$ be a connected compact orientable surface with $\chi(\surface)<0$, and suppose that $\Gamma < \Mod(\surface)$ is an adequately reduced subgroup. Then, $\Gamma$ is either a non-cyclic group with infinite girth or a virtually infinite-cyclic group, or a finite group; moreover, these alternatives are mutually exclusive. In particular, if $\surface$ is not a pair of pants, $\Mod(\surface)$ has infinite girth.
\end{prop}

\begin{proof}
Choose an integer $m \geq 3$. Suppose $\Gamma$ is an adequately reduced subgroup of $\Mod(\surface)$, and let $\calG=\{ \gamma_1, \cdots , \gamma_n \}$ be a generating set of $\Gamma$. We assume that $\Gamma$ is infinite and not virtually infinite-cyclic, and aim to show that it has infinite girth. In this case, the statement (2) of Theorem~\ref{trichotomy: adeq-red for conn} must be satisfied.

First, we know that there is a pseudo-Anosov element $\sigma \in \Gamma_{(m)} \lhd \Gamma$. For each $1 \leq j \leq n$ and $\varepsilon=\pm 1$, the conjugate $\gamma_j^\varepsilon \sigma \gamma_j^{-\varepsilon} \in \Gamma_{(m)} \lhd \Gamma$ is again a pseudo-Anosov element with $\Fix(\gamma_j^\varepsilon \sigma \gamma_j^{-\varepsilon})=\gamma_j^\varepsilon(\Fix(\sigma))$. Hence, applying the statement (2) of Theorem~\ref{trichotomy: adeq-red for conn} to the collection $\{\sigma\} \cup \{
\gamma_j^\varepsilon \sigma \gamma_j^{-\varepsilon} \; | \; 1 \leq j \leq n, \varepsilon=\pm1 \}$, we see that there is another pseudo-Anosov element $\tau \in \Gamma_{(m)} \lhd \Gamma$ such that $\Fix(\tau) \cap \Fix(\sigma)=\nil$ and $\Fix(\tau) \cap \Fix(\gamma_j^\varepsilon \sigma \gamma_j^{-\varepsilon})=\nil$ for all $1 \leq j \leq n$ and $\varepsilon=\pm 1$.

Note that, if $U_\sigma$ is a neighborhood of $\Fix(\sigma)$, then $\gamma_j^\varepsilon(U_\sigma)$ is a neighborhood of $\Fix(\gamma_j^\varepsilon \sigma \gamma_j^{-\varepsilon})$ for each $1 \leq j \leq n$ and $\varepsilon=\pm1$. It then follows that there are small enough neighborhoods $U_\sigma \supset \Fix(\sigma)$ and $U_\tau \supset \Fix(\tau)$ such that
\[
U_\sigma \cap U_\tau=\nil \;\; \text{and} \;\; U_\tau \cap \gamma_j^\varepsilon(U_\sigma)=\nil
\]
for each $1 \leq j \leq n$ and $\varepsilon=\pm1$, or equivalently
\[
U_\sigma \cap U_\tau=\nil \;\; \text{and} \;\; \gamma_j^\varepsilon(U_\tau) \cap U_\sigma=\nil
\]
for each $1 \leq j \leq n$ and $\varepsilon=\mp1$. Now, since $\sigma$ and $\tau$ are pseudo-Anosov elements, we can take high enough powers $\tilde \sigma:=\sigma^N$ and $\tilde \tau:=\tau^N$ such that
\[
\tilde \sigma^k \big( \PML(\surface)-U_\sigma \big) \subset U_\sigma \;\; \text{and} \;\; \tilde \tau^k \big( \PML(\surface)-U_\tau \big) \subset U_\tau
\]
for all non-zero integer $k$. In particular, we have
\[
\displaystyle{ \tilde \sigma^k \bigg( U_\tau \cup \bigcup_{\varepsilon=\pm1} \bigcup_{j=1}^n \gamma_j^\varepsilon(U_\tau) \bigg) \subset U_\sigma }
\;\; \text{and} \;\;
\displaystyle{ \tilde \tau^k \bigg( U_\sigma \cup \bigcup_{\varepsilon=\pm1} \bigcup_{j=1}^n \gamma_j^\varepsilon(U_\sigma) \bigg) \subset U_\tau }
\]
for all non-zero integer $k$. Applying the Infinite Girth Criterion (Proposition~\ref{igc}) to $\tilde \sigma$, $\tilde \tau$, $U_\sigma$, $U_\tau$, and
\[
x \in \PML(\surface)-\bigg((U_\sigma \cup U_\tau) \cup \bigcup_{\varepsilon=\pm1} \bigcup_{j=1}^n \gamma_j^\varepsilon(U_\sigma \cup U_\tau)\bigg)
\]
we conclude that $\Gamma$ must be a non-cyclic group with infinite girth.
\end{proof}

Now, we allow $\surface$ to be disconnected, and consider an adequately reduced subgroup $\Gamma < \Mod(\surface)$. The idea of the proof of girth alternative in this case is to consider the partition $\surface=\surface^{[0]} \sqcup \surface^{[1]} \sqcup \surface^{[2]}$ from Lemma~\ref{partition: adeq-red}, and restrict the group $\Gamma$ to $\surface^{[2]}$ when $\surface^{[2]} \neq \nil$. The non-emptiness of $\surface^{[2]}$ is the source of the existence of non-abelian free subgroup in the Tits alternative, as well as the infinite girth in the girth alternative. When $\surface^{[2]}$ is empty, the following lemma from the proof of the Tits alternative shows that $\Gamma$ must be virtually free-abelian.

\begin{lem}[{c.f. \cite[\S8.7]{Ivanov:MCGBook}}]
\label{vAb: adeq-red}
Let $m \geq 3$ be an integer, and let $\Gamma$ be an adequately reduced group. $\Gamma$ is virtually free-abelian if and only if the restriction of $\Gamma_{(m)}$ to each component of $\surface$ is either trivial or infinite-cyclic, i.e. $\surface^{[2]}=\nil$ in the partition from Lemma~\ref{partition: adeq-red}.
\end{lem}

\begin{prop} \label{girth-alt: adeq-red for disconn}
Let $\surface=\bigsqcup_{i=1}^{c(\surface)} \surface^i$ be a (possibly disconnected) compact orientable surface with $\chi(\surface^i)<0$ for all $i$, and suppose that $\Gamma < \Mod(\surface)$ is an adequately reduced subgroup. Then, $\Gamma$ is either a non-cyclic group with infinite girth or a virtually free-abelian group; moreover, these alternatives are mutually exclusive.
\end{prop}

\begin{proof}
Consider the decomposition $\surface=\surface^{[0]} \sqcup \surface^{[1]} \sqcup \surface^{[2]}$ in Lemma~\ref{partition: adeq-red}. If $\surface^{[2]}=\nil$, then $\Gamma$ is virtually free-abelian by Lemma~\ref{vAb: adeq-red}. Hence, we may assume that $\surface^{[2]} \neq \nil$. Note that the restriction $\Gamma^{[2]}$ of $\Gamma$ to $\surface^{[2]}$ is the image of a homomorphism from $\Gamma$ into $\Mod(\surface^{[2]})$. If $\Gamma^{[2]}$ is a non-cyclic group with infinite girth, $\Gamma$ itself must also be a noncyclic group with infinite girth by Proposition~\ref{akhmedov criterion}. Hence, we may as well assume that $\surface=\surface^{[2]}$ and $\Gamma=\Gamma^{[2]}$.

Choose $m \geq 3$. Let $\calG=\{\gamma_1, \cdots, \gamma_n\}$ be a generating set of $\Gamma$. By Theorem~\ref{trichotomy: adeq-red for disconn}, (i) there is a pseudo-Anosov element $\sigma \in \Gamma_{(m)} \lhd \Gamma$, and (ii) there is another pseudo-Anosov element $\tau \in \Gamma_{(m)} \lhd \Gamma$ such that $\Fix(\tau) \cap \Fix(\sigma)=\nil$ and $\Fix(\tau) \cap \Fix(\gamma_j^\varepsilon \sigma \gamma_j^{-\varepsilon})$ for each $1 \leq j \leq n$ and $\varepsilon=\pm1$. Here, the conjugates $\gamma_j^\varepsilon \sigma \gamma_j^{-\varepsilon}$ are pseudo-Anosov elements with $\Fix(\gamma_j^\varepsilon \sigma \gamma_j^{-\varepsilon})= \gamma_j^\varepsilon(\Fix(\sigma))$ in $\PML^\sharp(\surface)$. The arguments identical to the proof of Proposition~\ref{girth-alt: adeq-red for conn} --- with the space $\PML(\Sigma)$ replaced by $\PML^{\sharp}(\surface)$ --- completes the proof of the proposition.
\end{proof}

\subsection{Girth of Reducible Subgroups}
\label{Girth: Red}

We now consider the girth of a reducible group $\Gamma$. We take the canonical reduction system $\calC$, and consider the canonical reduction $\red_{\calC}(\Gamma)<\Mod(\surface_{\calC})$, for which the girth alternative holds by Proposition~\ref{girth-alt: adeq-red for disconn}. The following lemma, extracted from the proof of the Tits alternative \cite[\S8.9]{Ivanov:MCGBook}, characterizes virtually free-abelian subgroups $\Gamma$ in terms of its reduction $\red_{\calC}(\Gamma)$.




\begin{lem} [{c.f. \cite[\S8.9]{Ivanov:MCGBook}}]
\label{vAb: red}
$\Gamma$ is virtually free-abelian if and only if the canonical reduction $\red_{\calC}(\Gamma)$ is virtually free-abelian.
\end{lem}

\begin{prop} \label{girth-alt: red}
Let $\surface=\bigsqcup_{i=1}^{c(\surface)}\surface^i$ be a compact orientable surface with $\chi(\surface^i)<0$, and suppose that $\Gamma < \Mod(\surface)$ is a reducible subgroup. Then, $\Gamma$ is either a non-cyclic group with infinite girth or a virtually free-abelian group; moreover, these alternatives are mutually exclusive.
\end{prop}

\begin{proof}
Let $\calC$ be the canonical reduction system for $\Gamma$ and $\red_{\calC}$ be the corresponding reduction homomorphism. The reduction $\red_{\calC}(\Gamma)$ is adequately reduced, and hence by Proposition~\ref{girth-alt: adeq-red for disconn}, $\red_{\calC}(\Gamma)$ is either a non-cyclic group with infinite girth or a virtually free-abelian group. If $\red_{\calC}(\Gamma)$ is a non-cyclic group with infinite girth, noting that $\red_{\calC}(\Gamma)$ is the image of the homomorphism $\red_{\calC}:\Gamma \rightarrow \Mod(\surface_{\calC})$, we see that $\Gamma$ is also a non-cyclic group with infinite girth by Proposition~\ref{akhmedov criterion}. If $\red_{\calC}(\Gamma)$ is virtually free-abelian, so is $\Gamma$ by Lemma~\ref{vAb: red}.
\end{proof}

\bibliography{../../../MyTeX/MyMaster}
\bibliographystyle{../../../MyTeX/hamsalpha}

%
\end{document}